\documentclass[a4paper,10pt]{amsart}
\usepackage{amsmath,amssymb,amsfonts,amsthm}
\usepackage{latexsym,mathrsfs}
\usepackage{graphicx}
\usepackage[all]{xy}
\input xypic
\usepackage{color}
\usepackage{hyperref}
\hypersetup{colorlinks=true,linkcolor=red,citecolor=blue}
\usepackage[T1]{fontenc}
\usepackage[utf8]{inputenc}
\usepackage{lmodern}
\usepackage{microtype}
\theoremstyle{plain}
\newtheorem{theorem}[subsection]{{\bf Theorem}}
\newtheorem*{theorem*}{{\bf Theorem}}

\newtheorem*{corollary*}{{\bf Corollary}}
\newtheorem{proposition}[subsection]{{\bf Proposition}}
\newtheorem{lemma}[subsection]{{\bf Lemma}}

\theoremstyle{definition}

\theoremstyle{remark}

\numberwithin{equation}{subsection}

\newcommand{\N}{\mathcal{N}}
\newcommand{\B}{\mathcal{B}}
\newcommand{\V}{\mathcal{V}}
\newcommand{\A}{\mathcal{A}}

\begin{document}
\title[Finite $p$-groups in varieties]{On finite $p$-groups satisfying given laws}
\author{Primo\v z Moravec}
\address{{
Faculty of Mathematics and Physics\\
University of Ljubljana \\
Slovenia}}
\email{primoz.moravec@fmf.uni-lj.si}
\thanks{The author acknowledges the financial support from the Slovenian Research Agency (research core funding No. P1-0222, and projects No. J1-8132, J1-7256 and N1-0061).}
\date{\today}
\begin{abstract}

A variety of groups does not contain all metabelian groups if and only if there is an absolute bound for the nilpotency classes of powerful $p$-groups in the given variety. Similarly, a variety contains only finitely many finite $p$-groups of any given coclass if and only if not every group that is an extension of an abelian group by an elementary abelian $p$-group belongs to that variety.

\end{abstract}
\maketitle
\section{Introduction}
\label{s:intro}

\noindent
There is a plethora of results in the theory of varieties of groups describing varieties with no subvarieties of specific type. One of these is the following result of Groves:

\begin{theorem}[Groves \cite{Grov71,Grov72}]
  \label{t:groves}
  Let $\V$ be a solvable variety of groups.
  \begin{enumerate}
    \item If $\V$ does not contain the variety $\A\A$ of all metabelian groups, then $\V$ is finite-exponent-by-nilpotent-by-finite-exponent.
    \item If $\V$ does not contain $\A\A_p$ for every prime $p$, then $\V$ is finite-exponent-by-nilpotent.
    \item  If $\V$ does not contain $\A_p\A$ for every prime $p$, then $\V$ is nilpotent-by-finite-exponent.
  \end{enumerate}
\end{theorem}

Here $\A$ denotes the variety of all abelian groups, $\A_p$ the variety of all elementary abelian $p$-groups, and the notations regarding the calculus of varieties closely follow Hanna Neumann's book \cite{Neu67} here and throughout the paper. As the varieties $\A\A$, $\A\A_p$ and $\A_p\A$ are generated by $C_\infty\wr C_\infty$, $C_\infty\wr C_p$ and $C_p\wr C_\infty$, respectively, the conditions on $\V$ in the above Groves' theorem are straightforward to verify, given the laws of $\V$.

Traustason \cite{Tra05} extended the above result to arbitrary varieties as follows:

\begin{theorem}[Traustason \cite{Tra05}]
\label{t:traustason}
Let $\V$ be a variety of groups.
\begin{enumerate}
  \item $\A\A\not\subseteq\V$ if and ony if there exists a constant $c$ such that every group in $\V$ which is residually a finite $p$-group for every prime $p$ is nilpotent of class $\le c$.
  \item  $\A_p\A\not\subseteq\V$ for every prime $p$ if and only if
  there exist constants $c$ and $e$ such that every finitely generated virtually solvable group in $\V$ belongs to $\N_c\B_e$.
\end{enumerate}
\end{theorem}

Here $\N_c$ stands for the variety of all groups that are nilpotent of class $\le c$, and $\B_e$ is the variety of all groups of exponent dividing $e$. Note that the varieties $\V$ with the property that
$\A_p\A\not\subseteq\V$ for every prime $p$ are called {\it weak generalized Burnside varieties} by Traustason \cite{Tra11}, or {\it J-varieties} by Segal \cite{Seg09}, who derives further characterizations of these.
It easily follows from Theorem \ref{t:traustason} or \cite[Theorem 4.8.2]{Seg09} that J-varieties can be characterized by their finite $p$-groups:

\begin{proposition}
  \label{p:Jvar}
  A variety $\V$ is a J-variety if and only if there exist constants $c$ and $e$ such that, for every prime $p$, every finite $p$-group in $\V$ belongs to $\N_c\B_e$.
\end{proposition}

In this paper we find new characterizations of varieties not containing $\A\A$, or $\A\A_p$ for all primes $p$, in terms of their finite $p$-groups.

As for the first case, recall that a finite $p$-group is said to be {\it powerful} if $p$ is odd and $G'\le G^p$, or $p=2$ and $G'\le G^4$. Powerful $p$-groups were first systematically studied by Lubotzky and Mann \cite{Lub87}.
A more general notion is that of PF $p$-groups introduced by  Fern\'{a}ndez-Alcober, Gonz\'{a}lez-S\'{a}nchez, and  Jaikin-Zapirain \cite{Fer08}. A finite $p$-group is said to be  a {\it PF-group} if it has a
{\it potent filtration}, that is, a central series $G=G_1\ge G_2\ge\cdots \ge G_k=1$ with
$[G_i,{}_{p-1}G]\le G_{i+1}^p$ for all $i$. Our first result is the following:

\begin{theorem}
  \label{t:powerful}
  Let $\V$ be a variety of groups. The following are equivalent.
  \begin{enumerate}
    \item $\V$ does not contain the variety $\A\A$ of all metabelian groups.
    \item There exists a constant $c$ depending on $\V$ only such that, for every prime $p$, the nilpotency classes of finite powerful $p$-groups in $\V$ are bounded by $c$.
    \item There exists a constant $c$ depending on $\V$ and $p$ only such that, for every prime $p$, the nilpotency classes of finite PF $p$-groups in $\V$ are bounded by $c$.
  \end{enumerate}
\end{theorem}

Theorem \ref{t:powerful} follows fairly straightforward from Traustason's work on Milnor groups \cite{Tra05}. The details are given in Section \ref{s:powerful}.
In particular, Theorem \ref{t:powerful} generalizes a result by Abdollahi and Traustason \cite{Abd02} stating that the nilpotency class of $n$-Engel powerful $p$-groups can be bounded in terms of $n$ only. A consequence of the above theorem is also the fact that $\A\A\not\subseteq \V$ if and only if there is a constant $c$ such that for every prime $p$, every powerful pro-$p$ group in $\V$ is nilpotent of class $\le c$.

The second main result of this paper deals with finite $p$-groups of fixed coclass belonging to a variety $\V$ that does not contain $\A\A_p$ for every prime $p$.
Let $G$ be a group of order $p^n$, $n\ge 3$, and nilpotency class $c$. Then $r=n-c$ is called the {\it coclass} of $G$. Using coclass as the main invariant to study finite $p$-groups has proved to be very fruitful.
Here we prove:

\begin{theorem}
  \label{t:coclass}
  Let $\V$ be a variety of groups. The following are equivalent.
  \begin{enumerate}
    \item $\V$ does not contain the variety $\A\A_p$ for every prime $p$.
    \item For every prime $p$ and positive integer $r$, there are only finitely many finite $p$-groups of coclass $r$ that belong to $\V$.
    \item Every pro-$p$ group of finite coclass that belongs to $\V$ is finite.
  \end{enumerate}
\end{theorem}

\section{Powerful $p$-groups}
\label{s:powerful}

\noindent
Let $N$ be a normal subgroup of a finite $p$-group $G$. According to \cite{Lub87}, we say that $N$ is {\it powerfully embedded} in $G$ if $p$ is odd and $[N,G]\le N^p$, or $p=2$ and $[N,G]\le N^4$. Similarly, we say that $N$ is {\it PF-embedded} in $G$ if there exists a potent filtration of $G$ going through $N$, cf \cite{Fer08}.

\begin{lemma}[\cite{Lub87, Fer08}]
  \label{l:pepf}
  Let $G$ be a finite $p$-group and $N$ a normal subgroup of $G$ that is powerfully embedded (PF-embedded) in $G$.
  \begin{enumerate}
    \item $[N,G]$ and $N^p$ are powerfully embedded (PF-embedded) in $G$.
    \item $[N^{p^i},G^{p^j}]^{p^k}=[N,G]^{p^{i+j+k}}$.
    \item $N^{p^i}=\{ n^{p^i}\mid n\in N\}$.
    \end{enumerate}
    \end{lemma}

\begin{lemma}
  \label{l:powclass}
  Let $G$ be a finite $p$-group and suppose that $\gamma _{c+1}(G)^{p^e}=1$.
  \begin{enumerate}
    \item If $G$ is a powerful group, then it is nilpotent of class $\le c+e$.
    \item If $G$ is a PF-group, then it is nilpotent of class $\le c+e(p-1)$.
  \end{enumerate}
\end{lemma}

\proof
We prove only the second part, the first one follows along the same lines.
Let $G$ be a PF $p$-group. Then Lemma \ref{l:pepf} implies that $\gamma _{c+1}(G)$ is PF-embedded in $G$. Thus there exists a series
$\gamma_{c+1}(G)=N_1\ge N_2\ge\cdots \ge N_k=1$, where $N_i$ are normal subgroups of $G$, $[N_i,G]\le N_{i+1}$ and $[N_i,{}_{p-1}G]\le N_{i+1}^p$ for all $i$.
By definition, $\gamma_{c+p}(G)\le N_2^p$, and it follows that $\exp\gamma_{c+p}(G)$ divides $p^{e-1}$. Induction now shows that
$\gamma _{c+i(p-1)+1}(G)^{p^{e-i}}=1$, therefore
$\gamma _{c+e(p-1)+1}(G)=1$, as required.
\endproof

Let $G$ be a group and let $f[X]\in \mathbb{Z}[X]$ be a polynomial. Then we say \cite{Tra05} that $G$ is {\it $f$-Milnor} if
$$a^{f(t)}\equiv 1 \mod \left ( \langle a\rangle ^{\langle t\rangle}\right )^{\rm ab}$$
for all $a,t\in G$.

\begin{proof}[Proof of Theorem \ref{t:powerful}]
  Let $\V$ be a variety that does not contain all metabelian groups. By \cite[Theorem 4.9]{Tra05}, there exists a non-zero polynomial $f[X]\in\mathbb{Z}[X]$ such that all groups in $\V$ are $f$-Milnor. By \cite[Theorem 3.19]{Tra05} there exist constants $c=c(f)$ and $e=e(f)$ such that $\gamma _{c+1}(G^e)=1$ for all finite groups $G\in\V$. If $G$ is a powerful $p$-group (or a PF $p$-group) in $\V$,
  then the latter implies $\gamma _{c+1}(G)^{p^v}=1$, where $p^v$ is the largest $p$-power dividing $e$. By Lemma \ref{l:powclass}, the class of $G$ is thus bounded by a constant depending only on $\V$ (and $p$ in the case of PF $p$-groups). This shows that (1) implies (2) and (3).

  As all finite metacyclic $p$-groups, where $p$ is odd, are powerful \cite{Lub87}, it is clear that (2) implies (1), and similarly (3) implies (1). This completes the proof.
\end{proof}

\section{Finite $p$-groups of fixed coclass}
\label{s:coclass}

\noindent
Before proving Theorem \ref{t:coclass} we recall some concepts and results on coclass. We refer to \cite{Lee02} and \cite[Chapter 10]{Dix91}.

A fundamental tool in coclass theory is the so-called {\it coclass graph} $\Gamma(p,r)$ constructed as follows. The vertices of $\Gamma(p,r)$ correspond to isomorphism types of finite $p$-groups of coclass $r$. Two groups $G$ and $H$ in $\Gamma(p,r)$ are connected by a directed edge $G\to H$ if there exists a normal subgroup $N$ of $H$ with $|N|=p$ and $H/N\cong G$. In this case, $N$ is clearly the last non-trivial term of the lower central series of $H$.

Let $G$ and $H$ be two vertices of $\Gamma(p,r)$. We say that $H$ is a {\it descendant} of $G$ if there is a directed path from $G$ to $H$ in  $\Gamma(p,r)$. The subgraph $\mathcal{T}_G(p,r)$ of $\Gamma(p,r)$ generated by all descendants of $G$ is a tree; it is called a {\it coclass tree} if it contains precisely one infinite path starting with $G$ and is maximal with respect to this property. One of the key results in the theory of coclass is that $\Gamma(p,r)$ consists of finitely many coclass trees and finitely many groups not belonging to coclass trees.

If $\mathcal{T}=\mathcal{T}_G(p,r)$ is a coclass tree with the root $G$, then the unique infinite path in $\mathcal{T}$ starting with $G$ is called the {\it main line} of $\mathcal{T}$. The groups on the main line form an inverse system of finite $p$-groups. Its inverse limit is an infinite pro-$p$ group of coclass $r$. It turns out that infinite pro-$p$ groups of coclass $r$ correspond to the main lines of coclass trees in $\Gamma (p,r)$, and thus there is only finitely many of them.

Let $\mathcal{T}$ be a coclass tree, as above, with the main line
$G=G_1\to G_2\to\cdots$. Then the subgraph of $\mathcal{T}$ generated by all descendants of $G_i$ which are not descendants of $G_{i+1}$ is called the {\it $i$-th branch} $\mathcal{T}_i$ of $\mathcal{T}$. Clearly, $\mathcal{T}_i$ is a finite tree with root $G_i$.

\begin{proof}[Proof of Theorem \ref{t:coclass}]
  Note that every finite metabelian $p$-group of coclass 1 belongs to $\A\A_p$, see \cite[p. 52]{Lee02}. This shows that (2) implies (1).

  Suppose now that (1) holds. Fix a prime $p$ and positive integer $r$. Let $\V_{p,r}$ be the subvariety of $\V$ generated by all finite $p$-groups of coclass $r$ in $\V$. By \cite[Corollary 6.4.6 (Conjecture B)]{Lee02}, the derived length of $p$-groups of coclass $r$ can be bounded by an absolute constant depending on $p$ and $r$ only, therefore $\V_{p,r}$ is a solvable variety. By Theorem \ref{t:groves}, there exist constants $e$ and $c$, depending on $\V$, $p$ and $r$, such that $\V_{p,r}\subseteq \B_{p^e}\N_c$. Thus, if $G$ is a finite $p$-group of coclass $r$ that belongs to $\V$, then
  $\gamma _{c+1}(G)^{p^e}=1$.

  Assume first that $p$ is odd. We may assume that $|G|\ge p^{2p^r+r}$. Denote $m=p^r-p^{r-1}$ and $d=d(\gamma _m(G))$. Then a result of Shalev \cite[Theorem 1.2]{Sha94} implies that $\gamma_m(G)$ is powerful, $d=(p-1)p^s$ for some $0\le s\le r-1$, and $\gamma_i(G)^p=\gamma _{i+d}(G)$ for all $i\ge m$. Also, $\gamma _{2(m-1)}(G)$ is nilpotent of class $\le 2$ by \cite[Theorem 1.1]{Sha94},
  hence it is regular. Therefore, if $k\ge 2(m-1)$, then induction on $j$ shows that $\gamma _k(G)^{p^j}=\gamma _{k+jd}(G)$. Thus if $k\ge \max\{ c+1,2(m-1)\}$, then $\gamma _{k+ed}(G)=1$. This shows that the nilpotency class of $G$ is bounded in terms of $\V$, $p$ and $r$, and therefore the order of $G$ is bounded.

  If $p=2$, the proof that $|G|$ is bounded in terms of $r$ and $\V$ is similar. First, we may assume that $|G|\ge 2^{2^{2r+5}}$. Then $\gamma _{7\cdot 2^r-2}(G)$ is abelian \cite[Theorem 1.7]{Sha94}. Set $m=2^{r+2}$ and $d=d(\gamma _m(G))$. Then it follows from \cite[Theorem 1.5]{Sha94} that $\gamma _m(G)$ is powerful, $d\le 2^{r+1}$, and $\gamma _i(G)^2=\gamma _{i+d}(G)$ for all $i\ge m$. As above we obtain that $\gamma _{k+ed}(G)=1$ for all $k\ge \max\{ c+1, 7\cdot 2^r-2\}$, and hence the order of $G$ is bounded in terms of $r$ and $\V$.

  We have therefore showed that (2) implies (1). It remains to prove that (2) and (3) are equivalent. It is obvious that (2) implies (3). Conversely suppose that (3) holds, and consider an infinite pro-$p$ group $S$ of coclass $r$. The corresponding coclass tree has the main line $S_i\to S_{i+1}\to\cdots$, where $S_i=S/\gamma _i(S)$ and $i$ is large enough. As $S\notin\V$, there exists an integer $i_0\ge i$ such that $S_j\notin\V$ for all $j\ge i_0$. Let $\mathcal{T}_j$ be the branch of $\mathcal{T}$ with root $S_j$. If $G\in\mathcal{T}_j$ and $j\ge i_0$, then $G\not\in \V$. This shows that only finitely many vertices of a coclass tree belong to $\V$. As there are only finitely many coclass trees in $\Gamma(p,r)$ and finitely many $p$-groups of coclass $r$ outside these coclass trees, it follows that (3) implies (2). The proof is complete.
\end{proof}

\end{document}